\documentclass[10pt]{article}
\usepackage{amsmath,amsthm,amssymb,epsfig}

\newcommand{\refpart}[1]{{\it (#1)}}  

\newtheorem{theorem}{Theorem}[section]
\newtheorem{lemma}[theorem]{Lemma} 
\newtheorem{propose}[theorem]{Proposition} 
\newtheorem{corollary}[theorem]{Corollary} 
\newtheorem{definition}[theorem]{Definition} 
\newtheorem{remark}[theorem]{Remark} 

\newcommand{\PP}{\mathbb{P}}
\newcommand{\FF}{\mathbb{F}}
\newcommand{\RR}{\mathbb{R}}
\newcommand{\CC}{\mathbb{C}}
\newcommand{\QQ}{\mathbb{Q}}
\newcommand{\ZZ}{\mathbb{Z}}
\newcommand{\XX}{{\cal X}}
\newcommand{\KC}{{\cal K}}
\newcommand{\KG}{\Lambda}
\newcommand{\ia}{\alpha}
\newcommand{\ib}{\beta}
\newcommand{\PSL}{\mbox{\rm PSL}}

\newcommand{\hpg}[5]{{}_{#1}\mbox{\rm F}_{\!#2}\!
  \left(\left.{#3 \atop #4}\right| #5 \right) }

\newcommand{\hpgo}[2]{{}_{#1}\mbox{\rm F}_{\!#2}}

\title{Dihedral evaluations of hypergeometric functions with the Kleinian projective monodromy}

\author{
        Raimundas Vidunas\footnote{
        Vilnius University, Lithuania.
        E-mail: {\sf rvidunas@gmail.com}.}
       }
       
\date{}       

\begin{document}

\maketitle

\begin{abstract}
Algebraic hypergeometric functions can be compactly expressed as radical or dihedral functions
on pull-back curves where the monodromy group is much simpler. 
This article considers the classical 
$\hpgo32$-functions with the projective monodromy group PSL$(2,\FF_7)$
and their pull-back transformations of degree 21 that reduce the projective monodromy
to the dihedral group $D_4$ of 8 elements.
\end{abstract}

\section{Introduction}

One way to obtain a workable expression for an algebraic hypergeometric function
is to pull-back it to an algebraic curve where the (finite) monodromy group would be simpler, 
say, a finite cyclic group \cite{ViDarb}.
For example,
\begin{align} \label{eq:octa1} 
\hpg{2}{1}{5/24,\;13/24}{5/4}{\frac{108x\,(x-1)^4}{(x^2+14x+1)^3}}
 & =  \frac{1}{1-x}\,\left(1+14x+x^2\right)^{5/8}, \\
\hpg{2}{1}{1/6,\;5/6}{5/4}{\frac{27x\,(x+1)^4}{2(x^2+4x+1)^3}\,} 
& = \frac{\big(1+2x\big)^{1/4}}{1+x}\,\sqrt{1+4x+x^2} \quad
\end{align}
around $x=0$.
Here the $\hpgo21$-functions have the octahedral group $\cong S_4$ 
as the projective monodromy group (of the hypergeometric differential equation). 
The rational arguments of degree 6 reduce the monodromy to small cyclic groups,
as evidenced by the radical (i.e., algebraic power) functions on the right-hand sides of these identities.

If a Fuchsian differential equation $E$ on the Riemann sphere $\CC\PP^1$ has an algebraic solution $f$, 
then $E$ can be transformed by a pull-back transformation with respect 
to an algebraic covering $\varphi:B\to\CC\PP^1$ so that $f$ 
becomes a rational or radical solution on the curve $B$.  
The monodromy representation of the transformed equation then
has an invariant subspace generated by $f$, 
and $\varphi$ is 
a {\em Darboux covering} as defined in \cite{ViDarb}.
The explicit expression of $f$ as a radical 
function on $B$ is called a {\em Darboux evaluation} of $f$.
In \cite{ViDarb}, all tetrahedral, octahedral and icosahedral Schwarz types \cite{Schwarz}
of algebraic $\hpgo21$-functions are exemplified by Darboux evaluations.

Reduction of a finite monodromy group to a dihedral (rather than cyclic) group is worth attention as well.
The degree of the pull-back covering would be generally smaller, and dihedral expressions are still 
compact and practically workable. For example, 
a dihedral expression of octahedral function (\ref{eq:octa1})
is obtained after a cubic transformation \cite[(21)]{ViAGH}:
\begin{align}
\hpg{2}{1}{5/24,\;13/24}{5/4}{\frac{27x\,(x-1)^2}{(3x+1)^3}} 
& =(1+3x)^{5/8}\,\hpg{2}{1}{5/8,\;9/8}{5/4}{x}  \nonumber \\
& = \frac{(1+3x)^{5/8}}{\sqrt{1-x}}\,\hpg{2}{1}{5/8,\;1/8}{5/4}{x} \\
& = \frac{(1+3x)^{5/8}}{\sqrt{1-x}}\,\left(\frac{1+\sqrt{1-x}}2\right)^{\!-1/4}\!. \quad\nonumber
\end{align}
Here the second $\hpgo21$-function has a dihedral monodromy group, 
and is converted using standard transformations  \cite[(17), (2)]{ViAGH}. 
It can be evaluated directly using \cite[(3.1) with $k=-1$]{VidunasDh}.


Algebraic generalized hypergeometric functions $\hpgo{p}{p-1}$ are classified by Beukers and Heckman
\cite{BeuHeck}. One particularly interesting case \cite{Kato11}, \cite{vdPut00} 
is algebraic $\hpgo32$-functions such that the projective monodromy group 
(of their third order Fuchsian equations) is the simple group 
\begin{equation}
\KG=\PSL(2,\FF_7)\cong \mbox{\rm GL}(3,\FF_2)
\end{equation}
with 168 elements. Third order Fuchsian equations with this projective monodromy group
were anticipated by Klein \cite[a footnote in \S9]{Klein79},
and first constructed by Halphen \cite{Halphen84} and Hurwitz \cite{Hurwitz86}.

In \cite[Table 8.3]{BeuHeck}, classes of 
$\hpgo32$-functions with the projective monodromy group $\KG$ are labelled by the numbers 2, 3, 4.
The customary monodromy group (inside ${\rm GL}(3,\CC)$) of their Fuchsian equations 
is the complex reflection group ST24 in the Shephard--Todd classification \cite{STrg},
isomorphic to the central extension $\KG\times(\ZZ/2\ZZ)$.
A classification up to contiguous relations of (Fuchsian equations for) $\hpgo32$-functions 
with the projective monodromy group $\KG$ is given in \cite[Proposition 2.1]{Vid18b}.
It give these six classes and representative $\hpgo32$-functions:
\begin{align} \label{eq:class347}
{\rm (3A)}:  & \quad \hpg32{-\frac{3}{14},\frac{1}{14},\frac9{14}}{\frac13,\;\frac23}{z}; &
{\rm (3B)}:  & \quad  \hpg32{-\frac{1}{14},\frac{3}{14},\frac5{14}}{\frac13,\;\frac23}{z}; \nonumber \\
{\rm (4A)}:  & \quad  \hpg32{-\frac{3}{14},\frac{1}{14},\frac9{14}}{\frac14,\;\frac34}{z}; &
{\rm (4B)}:  & \quad \hpg32{-\frac{1}{14},\frac{3}{14},\frac5{14}}{\frac14,\;\frac34}{z}; \quad\\
{\rm (7A)}:  & \quad \hpg32{-\frac{1}{14},\frac{1}{14},\frac5{14}}{\frac17,\;\frac57}{z}; &
{\rm (7B)}:  & \quad \hpg32{-\frac{1}{14},\frac{1}{14},\frac9{14}}{\frac27,\;\frac67}{z}. \nonumber
\end{align}
Equations of type (3A) are directly related to the modular curve $\XX(7)$, 
and to Klein's quadric curve 
\begin{equation} \label{eq:klein}
X^3Y+Y^3Z+Z^3X=0.
\end{equation}
This is a Riemann surface of genus $g=3$,
with the group of holomorphic symmetries isomorphic to $\KG$.

As shown in \cite{Vid18b}, the projective monodromy $\KG$ of considered Fuchsian equations
can be reduced to $\ZZ/7\ZZ$ by pull-back transformations of degree $\#\KG/7=24$.
The monodromy representation of transformed equations is completely reducible,
hence the pulled-back equations have a basis of radical (i.e., algebraic power) solutions.
This gives Darboux evaluations for all solutions of a considered Fuchsian equation
in terms of the basis radical solutions.

This article presents Darboux coverings of degree 21 that reduce the projective monodromy $\KG$
to the dihedral group $D_4$ with 8 elements. The dihedral group is a 2-Sylow subgroup of $\KG$;
see \cite[p.~66, 92]{Elkies}. The degree 21 covering exists by the Galois correspondence
associated to degree 168 Galois coverings with the monodromy $\KG$.

Let $E_0,E_1$ denote third order differential equations with the projective monodromies $\KG,D_4$, respectively, that and related by a pull-back transformation of degree 21. 
The customary monodromy group of $E_1$ is a central extension of $D_4$, thus a dihedral group as well. 
As there are no irreducible 3-dimensional representations of dihedral groups \cite{GroupRep}, 
the monodromy representation of $E_1$ is reducible. Its one-dimensional invariant
subspace gives a radical solution of $E_1$, hence a Darboux evaluation of a solution
of $E_0$. The two-dimensional invariant subspace of the monodromy of $E_1$ 
leads to {\em dihedral evaluations} of solutions of $E_0$.
This article presents Darboux and dihedral evaluations of representative 
$\hpgo32$-functions with respect to the degree 21 Darboux coverings.

\section{Preliminaries}

Section \ref{sec:hgp} recalls basic knowledge about differential equations for $\hpgo32$-functions, 
their pull-back transformations, and contiguous relations. 
Section \ref{th:3f2cases} characterizes the classification \ref{eq:class347}
of $\hpgo32$-functions with the projective monodromy $\KG$.
Section \ref{sec:darbcovs} introduces Darboux coverings following \cite{ViDarb}
and \cite[\S 2.4]{Vid18b}.

\subsection{Hypergeometric functions}
\label{sec:hgp}

The hypergeometric function $\hpg32{\!\ia_1,\,\ia_2,\,\ia_3\!}{\ib_1,\,\ib_2}{z}$ 
satisfies the differential equation
\begin{align} \label{eq:hpgde}
\! \left(z\frac{d}{dz}+\ia_1\!\right) \! \left(z\frac{d}{dz}+\ia_2\!\right) \! 
\left(z\frac{d}{dz}+\ia_3\!\right) \! Y(z)  \hspace{132pt} \\
=\frac{d}{dz}  \! \left(z\frac{d}{dz}+\ib_1-1\!\right) \!
\left(z\frac{d}{dz}+\ib_2-1\!\right) \! Y(z), \nonumber
\end{align}
minding the commutativity rule $\frac{d}{dz}z=z\frac{d}{dz}+1$.
This is a third order Fucshian equation with three singular points $z=0$, $z=1$, $z=\infty$.
The singularities and local exponents at them are encoded by the generalized Riemann's $P$-symbol:
\begin{equation}
P\left( \begin{array}{ccc|c} 
z=0 & z=1 & z=\infty & \\ \hline
0 & 0 & \ia_1 & \\
1-\ib_1 & 1 & \ia_2 & z \\
1-\ib_2 & \gamma & \ia_3 &
\end{array}\right) 
\end{equation}
with $\gamma=\ia_1+\ia_2+\ia_3-\ib_1-\ib_2$. 
Generically, a basis of local solutions at $z=0$ and $z=\infty$ can be written in terms
of $\hpgo32$-series. Let $M=(m_{i,j})$ denote the matrix
\begin{equation} \label{eq:hypm}
M=\left( \begin{array}{ccc} 
\ia_1 & \ia_1-\ib_2+1 & \ia_1-\ib_1+1 \\
\ia_2 & \ia_2-\ib_2+1 & \ia_2-\ib_1+1 \\
\ia_3 & \ia_3-\ib_2+1 & \ia_3-\ib_1+1
\end{array}\right).
\end{equation}
A generic basis of local solutions at $z=0$ is
\begin{align}
& \hpg32{m_{1,1\,},\,m_{2,1\,},\,m_{3,1}}{\ib_1,\;\ib_2}{z}, \quad
z^{1-\ib_2}\,\hpg32{m_{1,2\,},\,m_{2,2\,},\,m_{3,2}}{2-\ib_2,\;\ib_1-\ib_2+1}{z}, \nonumber\quad \\
& z^{1-\ib_1}\,\hpg32{m_{1,3\,},\,m_{2,3\,},\,m_{3,3}}{2-\ib_1,\;\ib_2-\ib_1+1}{z},
\end{align}
while a generic basis of local solutions at $z=\infty$ is
\begin{align}
z^{-\ia_j}\,\hpg32{\!m_{j,1\,},\;m_{j,2\,},\;m_{j,3}}{\ia_j-\ia_{k}+1,\ia_j-\ia_{\ell}+1}{z}
\end{align}
with $ j\in\{1,2,3\}$ and $\{k,\ell\}=\{1,2,3\}\setminus\{j\}$.
We refer to the set of 6 functions 
formed by 3 local hypergeometric solutions (disregarding a power factor) at $z=0$ 
and 3 such local solutions at $z=\infty$ of the same third order Fuchsian equation 
as {\em companion hypergeometric functions} to each other.

Let $B$ denote an algebraic curve. Let $\varphi(\ldots)$ denote a rational function on $B$;
it defines an algebraic covering $\varphi:B\to\CC\PP^1$.
A {\em pull-back transformation} with respect to $\varphi$ 
of a differential equation for $y(z)$ in $d/dz$ has the form
\begin{equation} \label{algtransf}
z\longmapsto\varphi(\ldots), \qquad y(z)\longmapsto
Y(\ldots)=\theta(\ldots)\,y(\varphi(\ldots)),
\end{equation}
where 
$\theta(\ldots)$ is a radical function 
on $B$.  The equations that differ by a pull-back transformation with respect to 
the trivial covering \mbox{$\varphi(\ldots)=z$} 
are called {\em projectively equivalent}.

There are several algebraic transformations for $\hpgo32$-functions \cite{KatoTR}.
Here are quadratic and cubic transformations:
\begin{align} \label{eq:t32a}
\hpg32{a,a+\frac14,a+\frac12}{b+\frac14,3a-b+1}{-\frac{4z}{(z-1)^2}} \hspace{-40pt} \\  \displaybreak
& =(1-z)^{2a}\,\hpg32{2a,2a-b+\frac34,b-a}{b+\frac14,3a-b+1}{z}\!, \qquad \nonumber \\ 
\label{eq:t32c} \hpg32{a,a+\frac13,a+\frac23}{b+\frac12,3a-b+1}{\frac{27z^2}{(4-z)^3}\!} \hspace{-40pt} \\
& = \left(1-\frac{z}4\right)^{\!3a}\hpg32{3a,b,3a-b+\frac12}{2b,6a-2b+1}{z}\!.  \nonumber
\end{align}
They can be understood as pull-back transformations between Fuchsian equations
(particularly, (\ref{eq:hpgde})) for hypergeometric functions \cite{ViAGH}, \cite{KatoTR}.
The involved equations have the local exponent $c=1/2$  at $z=1$, 
and the quadratic or cubic arguments $\varphi(z)$
on the left-hand sides have properly branching points in the fiber $\varphi=1$. 
This helps the number of singularities of the pulled-back equation to equal merely 3.

Two $\hpgo32$-functions whose parameters $\ia_1,\ia_2,\ia_3,\ib_1,\ib_2$ 
differ respectively by integers 
are called {\em contiguous} to each other.
This defines a  {\em contiguity equivalence} relation on the $\hpgo32$-functions.
For example, differentiating a $\hpgo32$-function gives a contiguous function, generically:
\begin{equation}
\frac{d}{dz}\hpg32{\ia_1,\ia_2,\ia_3}{\ib_1,\ib_2}{z}=\frac{\ia_1\ia_2\ia_3}{\ib_1\ib_2}\,
\hpg32{\ia_1+1,\ia_2+1,\ia_3+1}{\ib_1+1,\ib_2+1}{z},
\end{equation}
Fuchsian equations of contiguous functions have the same monodromy, generically.
For a generic set of four contiguous $\hpgo32$-functions there is a linear {\em contiguous relation} 
between them \cite{Rainville}. 
For example, differential equation (\ref{eq:hpgde}) can be rewritten as a contiguous
relation between $\hpg32{\!\ia_1+n,\,\ia_2+n,\,\ia_3+n\!}{\ib_1+n,\;\ib_2+n}{z}$ with $n\in\{0,1,2,3\}$. 
Consequently, a contiguous function to a generic $\hpgo32$-function $F$ 
can be expressed linearly in terms of $F$ and its first and second derivatives (thus, as a gauge transformation).
In particular, we have 
\begin{align}
\hpg32{\ia_1+1,\ia_2,\ia_3}{\ib_1,\ib_2}{z} = & 
\left(1+\frac{z}{\ia_1}\,\frac{d}{dz}\right) \hpg32{\ia_1,\ia_2,\ia_3}{\ib_1,\ib_2}{z},\\
\hpg32{\ia_1,\ia_2,\ia_3}{\ib_1-1,\ib_2}{z} = & \left(1+\frac{z}{\ib_1-1}\,\frac{d}{dz}\right)
\hpg32{\ia_1,\ia_2,\ia_3}{\ib_1,\ib_2}{z},
\end{align}

\subsection{Monodromy groups}
\label{th:3f2cases}

Let $E$ denote a 3rd order Fuchsian equation on the Riemann sphere $\CC\PP^1$.
Suppose that $f_1,f_2,f_3$ is a basis of its solutions.
If either the (conventional) monodromy group 
or the differential Galois group \cite{vdPut00} 
of $E$ are finite, those two groups coincide with the classical Galois group of 
the finite field extension $\CC(z,f_1,f_2,f_3)\supset\CC$.
In that case, the {\em projective monodromy group} refers to
the the Galois group of the finite extension $\CC(z,f_2/f_1,f_3/f_1)\supset\CC(z)$.
Both extensions of $\CC(z)$ are Galois extensions, because the monodromy representation
gives linear transformations of $f_1,f_2,f_3$ (in GL$(3,\CC)$)
or fractional-linear transformations of $f_2/f_1,f_3/f_1$ (in PGL$(3,\CC)$).

Standard transformations that preserve the projective monodromy are:
\begin{itemize}
\item[\refpart{i}] Gauge transformations of the contiguity equivalence;
\item[\refpart{ii}] Projective equivalence transformations;
\item[\refpart{iii}] Being companion hypergeometric functions 
(i.e., being solutions of the same third order Fuchsian equation); 
\item[\refpart{iv}] Multiplying the hypergeometric parameters $\ia_1,\ia_2,\ia_3,\ib_1,\ib_2\in\QQ$ 
by an integer coprime to their denominators. 
\end{itemize} 
These transformations characterize the classification of algebraic hypergeometric 
functions in \cite[Theorem 7.1]{BeuHeck}.
 
The classification in (\ref{eq:class347}) of $\hpgo32$-functions with the projective monodromy $\KG$
is given up to the transformations \refpart{i}--\refpart{iii}; see \cite[Proposition 2.1]{Vid18b}.
Transformation \refpart{iv} relates the types (3A), (3B),
or the types {\rm (4A)}, {\rm (4B)}, or the types {\rm (7A)}, {\rm (7B)}.

\subsection{Darboux coverings}
\label{sec:darbcovs}

The notions of {\em Darboux curves}, {\em Darboux coverings} 
and {\em Darboux evaluations} are introduced in \cite[Ch.~4]{ViPhD} and \cite{ViDarb}.
The terminology is motivated by integration theory of vector fields \cite{DarbVF},
where {\em Darboux polynomials} determine invariant hypersurfaces.
In differential Galois theory \cite{Weil95}, Darboux polynomials are specified
by algebraic solutions of an associated Riccati equation.
Here is a formulation of \cite[Definition 3.1]{ViDarb}.
\begin{definition} 
Consider a linear homogeneous differential equation
\begin{equation} \label{eq:genlde}
\frac{d^n}{dz^n}+a_{n-1}(z)\frac{d^{n-1}}{dz^{n-1}}+\ldots
+a_1(z)\frac{d}{dz}+a_0(z)=0
\end{equation}
on $\CC\PP^1$, thus with $a_i(z)\in\CC(z)$. 
We say that an algebraic covering $\varphi:B\to\CC\PP^1$ is a {\em Darboux covering} for $(\ref{eq:genlde})$
if a pull-back transformation $(\ref{algtransf})$ of it with respect to $\varphi$ has a solution $Y$ such that:
\begin{enumerate}
\item[\refpart{a}] the logarithmic derivative $u=Y'/Y$ is 
a rational function on the algebraic curve $B$;
\item[\refpart{b}] the algebraic degree of $u$ over $\CC(z)$ equals the degree of $\varphi$.
\end{enumerate}
The algebraic curve $B$ is then called a {\em Darboux curve}.
\end{definition}
Condition \refpart{a} means that the monodromy representation of the pulled-back equation
has a one-dimensional invariant subspace (generated by $Y$). 
Determination of Darboux coverings is made easier by their basic properties.  
The following lemma underlines that Darboux coverings are ``invariant" 
under transformations of hypergeometric equations that preserve the monodromy.
\begin{lemma}
Let $E_1$ denote a hypergeometric equation $(\ref{eq:hpgde})$ with a finite primitive monodromy group.
Suppose that other hypergeometric equation $E_2$ is related to $E_1$ by 
transformations described in \refpart{i}, \refpart{iv} in \S $\ref{th:3f2cases}$.
If $\varphi:B\to\CC\PP^1$ is a Darboux covering for $E_1$,
then $\varphi$ is a Darboux covering for $E_2$ as well.
\end{lemma}
\begin{proof}
The transformations \refpart{i}, \refpart{iv} do not affect the primitive monodromy group,
thus equations $E_1,E_2$ have isomorphic monodromies.
Let $E_1^*$, $E_2^*$ denote the Fuchsian equations obtained from $E_1,E_2$, respectively,
by applying the same pull-back transformation with respect to $\varphi$.
The monodromies of $E_1^*$, $E_2^*$ are isomorphic. 
Therefore, if one has a radical solution so does the other.
\end{proof}

\begin{corollary} \label{th:darb3s}
The same degree $21$ Darboux coverings $\varphi$ or $1/\varphi$ 
for reduction of the projective monodromy $\KG=\rm {PSL}(2,\FF_7)$ 
of hypergeometric equation $(\ref{eq:hpgde})$
to the dihedral group $D_4$ apply to all hypergeometric functions of the types {\rm (3A)} and {\rm (3B)};
or to all  $\hpgo32$-functions of the types {\rm (4A)} and {\rm (4B)}; 
or to all  $\hpgo32$-functions of the types {\rm (7A)} and {\rm (7B)}.
\end{corollary}
\begin{proof}
Each of the six types describes an equivalence class under the contiguity equivalence.
The pairs of types {\rm (3A)}, {\rm (3B)}; or {\rm (4A)}, {\rm (4B)}; or {\rm (7A)}, {\rm (7B)}
are related by transformation \refpart{iv} in \S \ref{th:3f2cases}.
\end{proof}

To match hypergeometric functions with radical or dihedral solutions 
of a pulled-back Fuchsian equation, the next 
lemma is useful. It applies to 
equations obtained by the considered  degree 21 pull-back transformations. 
\begin{lemma} \label{th:locexps}
Suppose that differential equation $(\ref{eq:genlde})$ has a 
finite dihedral monodromy group.
If the local exponents $\lambda_1,\ldots,\lambda_n$ 
at a point $P\in\CC\PP^1$  are all different modulo $\ZZ$,
then for each local exponent $\lambda_i$  ($i\in\{1,\ldots,n\}$) 
there is exactly one (up to scalar multiplication) radical solution 
with the vanishing order $\lambda$ 
at $P$. 
\end{lemma}
\begin{proof} (Compare with \cite[Lemma 2.5]{Vid18b}.)
A monodromy representation of the finite dihedral group reduced to a direct sum 
of 1-dimensional and 2-dimensional invariant subspaces \cite{GroupRep}.
Each 1-dimensional subspace gives a radical solution, while each 2-dimensional space
gives two independent solutions of the form $(p\pm q\sqrt{g})^e$. 
This gives a basis of the whole solution space.
If two of those solutions had the same local exponent $\lambda^*$ at $P$,
their linear combination would have the vanishing order $\lambda^*+k$ with integer $k>0$.
Hence there is a basis solution for each of the $n$ local exponents.
\end{proof}

\section{Three Darboux coverings}
\label{sec:ourcovs}

Here we compute the Darboux coverings of degree 21 that reduce 
the projective monodromy group $\KG=\PSL(2,\FF_7)$
of third order hypergeometric equations (\ref{eq:hpgde}) 
to the dihedral group $D_4$ of 8 elements. 
They turn out to be Belyi maps \cite{LandoZvonkin}
with the following branching patterns (and genus $g$):
\begin{align}
\mbox{for the types {\rm (3A)} and \rm (3B):} & \qquad [7^3/3^7/2^{8}1^5] & (g=0); \\
\label{eq:brpat4} 
\mbox{for the types {\rm (4A)} and \rm (4B):} & \qquad  [7^3/4^42^21/2^{8}1^5] & (g=0); \\
\label{eq:brpat7} 
\mbox{for the types {\rm (7A)} and \rm (7B):} & \qquad  [7^3/7^3/2^{8}1^5] & (g=2).
\end{align}
As we will see, there is exactly one Darboux covering  $\Psi_3:\PP^1\to\PP^1$ 
(up to M\"obius transformations on either $\PP^1$) with the branching pattern $[7^3/3^7/2^{8}1^5]$. 
Just as the degree 24 Darboux coverings in \cite[\S 2.5]{Vid18b},
the degree 21 Darboux coverings 
are related to each other 
by quadratic and cubic transformations (\ref{eq:t32a})--(\ref{eq:t32c}) of hypergeometric solutions. 
Consequently, there are unique Darboux coverings  $\Psi_4:\PP^1\to\PP^1$ 
and $\Psi_7:H_7\to\PP^1$ (up to holomorphic symmetries of the Darboux curves and $\PP^1$) 
with the branching patterns in (\ref{eq:brpat4})--(\ref{eq:brpat7}), respectively.

Let $\KC$ denote Klein's curve (\ref{eq:klein}).
As is well known \cite[\S 5.1.1]{Kato11},
hypergeometric equations of type (3A) are directly related to this curve.
In particular, the identity \cite[Proposition 30]{FrancMason} 
\begin{align} \label{eq:kleinhpg} 
\hpg32{\frac5{42},\,\frac{19}{42},\,\frac{11}{14}}{\frac57,\;\frac87}{x}^{\!3}
\hpg32{\!-\frac1{42},\,\frac{13}{42},\,\frac{9}{14}}{\frac47,\;\frac67}{x} = \hspace{-103pt} \\ 
& \hpg32{\!-\frac1{42},\,\frac{13}{42},\,\frac{9}{14}}{\frac47,\;\frac67}{x}^{\!3}
\hpg32{\frac{17}{42},\,\frac{31}{42},\,\frac{15}{14}}{\frac97,\;\frac{10}7}{x}  \nonumber \\
& +\frac{x}{1728}\,\hpg32{\frac{17}{42},\,\frac{31}{42},\,\frac{15}{14}}{\frac97,\;\frac{10}7}{x}^{\!3}
\hpg32{\frac5{42},\,\frac{19}{42},\,\frac{11}{14}}{\frac57,\;\frac87}{x} \qquad \nonumber
\end{align}
gives a projective parametrization of $\KC$ by hypergeometric functions,
and the degree 168 Galois covering $\KC\to\PP^1$.
The Darboux coverings of types (3A), (3B) can be found by investigating degree 21 
subcoverings of the Galois covering. More concretely \cite[\S 4]{Elkies}, 
the modular curve $\XX(7)$ is isomorphic to $\KC$, 
and we may look at the Galois covering $\XX(7)\to\XX(1)$.
Its subcoverings are listed in the Cummins-Pauli tables \cite{CumminsP} 
as corresponding level 7 congruence subgroups. 
There we find the congruence subgroup 7D$^0$ 
of index 21, with $c_2=5$, $c_3=0$, and the cusps $7^3$.
This gives the branching pattern $[7^3/3^7/2^{8}1^5]$.   
Besides, the covering is unique (up to M\"obius transformations) and defined over $\QQ$,
as the table entry in \cite{CumminsP} indicates the conjugation orbit of size con$\;=1$.
There are no other entries of level 7 and index 21
(also no entries with the branching $[7^3/3^7/2^{10}1]$ of genus $g=1$).
Besides, 7D$^0$ is a subgroup of 7A$^0$. The larger congruence subgroup 7A$^0$
is of index 7, and gives the branching pattern $[2^21^3/3^21/7]$ (with con$\;=2$).
Hence, the degree 21 coverings are compositions of degree 7 Belyi maps
with this branching and cubic coverings.
The degree 7 maps are straightforward to compute \cite[Step 4 in \S 3]{ViAGH}.
They are defined over $\QQ(\sqrt{-7})$, say:
\begin{align}
\psi_7(z) & =\frac{z}{1728} \left( z+\sqrt{-7} \right)^3 \left(z+\frac{7+5\sqrt{-7}}{2}\right)^{\!3}.
\end{align}
Then $\psi_7(z)-1$ equals 
\begin{align}
\frac{z^2+(2\!+\!4\sqrt{-7})z-27}{1728} \! \left(z+\frac{9\!+\!5\sqrt{-7}}2\right) \!
\left(z^2+(2\!+\!2\sqrt{-7})z+\frac{-5\!+\!\sqrt{-7}}2\right)^{\!2}. \nonumber
\end{align}
To produce the branching $[7^3/3^7/2^{8}1^5]$, 
the cubic covering must branch above $z=0$ (with order 3)
and above 2 of the 3 simple roots of $\psi_7(z)-1$. 
The discriminant of $z^2+(2+4\sqrt{-7})z-27$ equals $16\sqrt{-7}$,
hence this polynomial does not factor over $\QQ(\sqrt{-7})$.
We must have simple branching points above its roots
to have a chance of obtaining a composition defined over $\QQ$ as expected.
A correct composition is
\begin{equation}
\Psi_3(x)=\psi_7\!\left(\frac{(4x-7-\sqrt{-7})^3}{(20-4\sqrt{-7})(x^3-7x+7)} \right).
\end{equation}
We obtain
\begin{equation} \label{eq:darbcov3}
\Psi_3=\frac{x^3(3x^2-7)^3(2x^2-7x+7)^3(11x^2-35x+28)^3}{1728(x^3-7x+7)^7}.
\end{equation}
A dessin d'enfant of this covering is depicted in Figure \ref{fig:dessins} \refpart{a}.
The covering $\Psi_3$ was computed in \cite[\S 4.3]{ViPhD} 
by lengthy computations with {\sf Maple} starting from the branching pattern alone. 
It is computed in \cite[(4.35)]{Elkies} as well,
with $x=2\phi/(\phi+1)$. 
(The factor $5\phi^2-15\phi-7$ there must be corrected to $5\phi^2-14\phi-7$.)

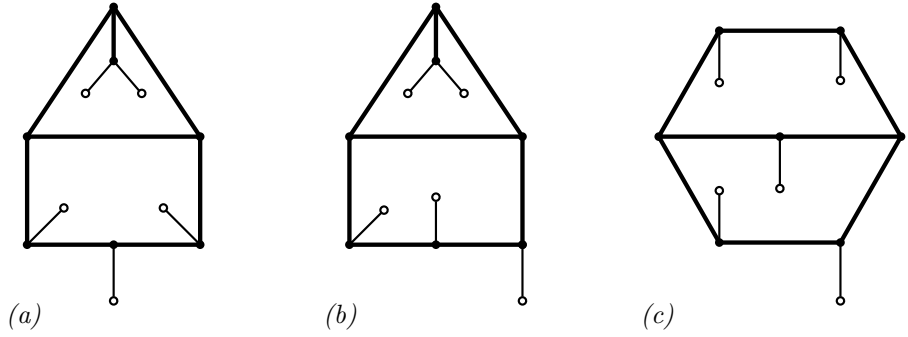
\begin{figure}
\begin{picture}(360,124)
\put(-10,2){\includegraphics[width=360pt]{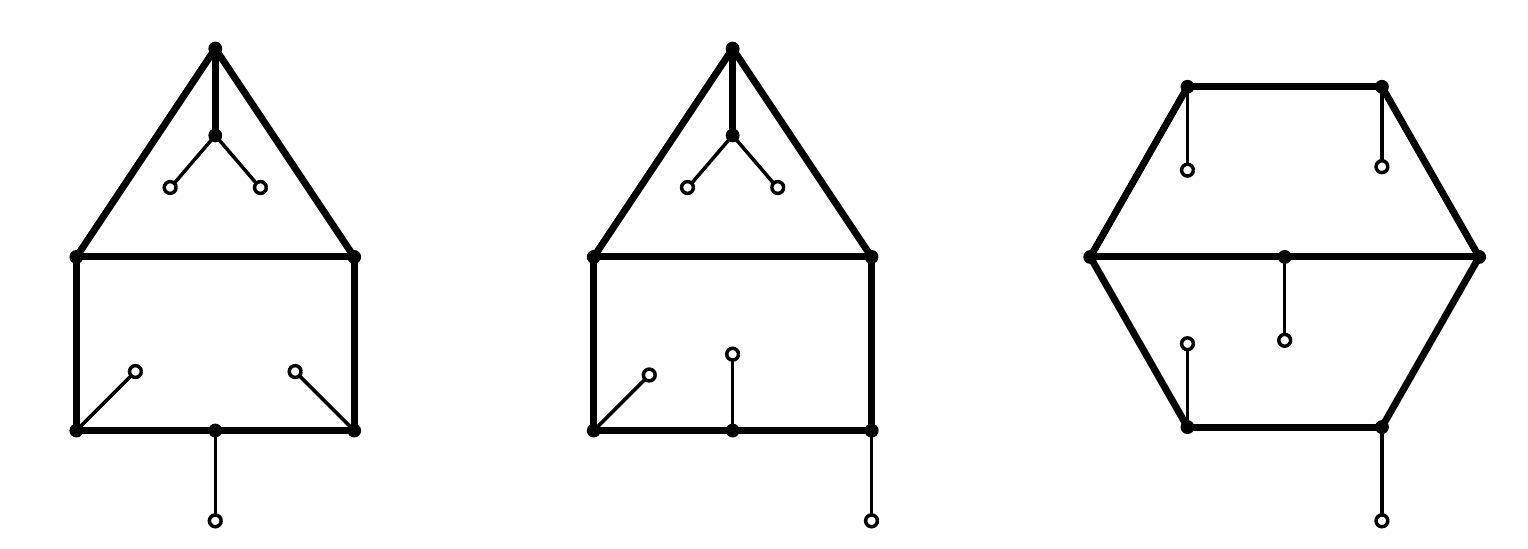}}
\put(0,0){\refpart{a}}  \put(120,0){\refpart{b}}  \put(240,0){\refpart{c}}
\end{picture}
\caption{The dessins d'enfant with the branching pattern $[7^3/3^7/2^{8}1^5]$.}
\label{fig:dessins}
\end{figure}

\begin{remark} \rm
Evidently, there is at least one other Galois orbit of Belyi maps 
with the same branching $[7^3/3^7/2^{8}1^5]$. 
They are alternative compositions of $\Psi_7$ with a cubic covering.
Those Belyi maps are defined over $\QQ\big(\sqrt{\sqrt{-7}}\,\big)$, 
that is, the splitting field of the polynomial $z^2+(2+4\sqrt{-7})z-27$.
The dessins d'enfant of these composite Belyi maps
are represented (up to mirroring) in Figure \ref{fig:dessins} \refpart{b}, \refpart{c}.
As in \cite{Hoeij}, thick edges represent two dessin edges with a white vertex of order 2 in the middle.
The Belyi maps are not quotients of the Galois covering $\KC\to\PP^1$
because their monodromy (i.e., the {\em cartographic group} \cite{LandoZvonkin}) 
is not homomorphic to $\KG$ but is a larger group of order $2^63^87$. 
Modern computations (following \cite{Hoeij}) show that there are no other Belyi maps
with the branching pattern $[7^3/3^7/2^{8}1^5]$.
\end{remark}
 
By Corollary \ref{th:darb3s}, the Belyi map $\Psi_3$ 
is a Darboux covering for type (3B) hypergeometric functions as well.
To compute degree 21 Darboux coverings for the other types, 
we follow their algebraic correspondence relations to $\Psi_3$ that are consequences 
of quadratic and cubic transformations 
of $\hpgo32$-functions \cite{KatoTR}.
In particular, representative $\hpgo32$-functions of types {\rm (3A)} and {\rm (7A)} are related 
by cubic transformation (\ref{eq:t32c}). For example,
\begin{equation} \label{eq:hpg37c}
\hpg32{\!-\frac1{42},\,\frac{13}{42},\,\frac{9}{14}}{\frac47,\;\frac67}{\frac{27z^2}{(4-z)^3}} = 
\left( 1-\frac{z}4\,\right)^{-1/14} \hpg32{\!-\frac1{14},\,\frac1{14},\,\frac{5}{14}}{\frac17,\;\frac57}{z}\!.
\end{equation}
The argument of the $\hpgo32$-function on the left equals $1/\Psi_3$  
in its Darboux or dihedral evaluation 
(so that $\Psi_3=0$ and $\Psi_3=\infty$ match correctly singularities of the hypergeometric equation).
The Darboux covering $\Psi_7$ 
is therefore constructed by parametrizing the fiber product
\begin{equation} \label{eq:fp7721}
\frac1{\Psi_3(x)}=\frac{27z^2}{(4-z)^3}
\end{equation}
(compare with \cite[Lemma 3.5]{ViDarb}). 
Irreducibility of this equation means that the Darboux covering $\Psi_7:H_7\to\PP^1$ is unique
(up to holomorphic symmetries of both curves).
To parametrize the fiber product by a simple equation, 
we substitute $z=8\hat{z}/(2\hat{z}+3)$ so that the right-hand side becomes $\hat{z}^{\,2}(2\hat{z}+3)$. 
After the next substitution $\hat{z}=\tilde{z}P_3^2/(xQ_3)$
we get the equation 
\begin{equation}
27P_3^3=49\tilde{z}^2(2\tilde{z}P_3^2+3xQ_3)
\end{equation}
of degree 9 in $x,\tilde{z}$.
This equation defines a curve of genus 2 isomorphic to 
\begin{equation} \label{eq:darbc7} \textstyle
v^2=u \, \big(1-\frac{63}4u+91u^2-231u^2+224u^4 \big).
\end{equation}
This is our standard model of $H_7$. 
Eventually, we can find this parametrization of (\ref{eq:fp7721}) by $H_7$: 
\begin{equation} \label{eq:cubpara37}
x=\frac{2v+(6u-1)(56u^2-21u+2)}{(4u-1)(56u^2-16u+1)}, \qquad z=\Psi_7,
\end{equation}
where $\Psi_7$ is the Darboux curve:
\begin{equation} \label{eq:darbcov7}
\Psi_7=
\frac{P_7^{\,2}\;(2v-3u+14u^2)^7}{u\,(1-4u)^7\,(-2v-3u+14u^2)^7}
\end{equation}
with
\begin{equation} \label{eq:poly47} \textstyle
P_7=v+\frac72u-35u^2+112u^3-128u^4.
\end{equation}
As a function on $H_7$, the Belyi map 
in (\ref{eq:fp7721}) has the branching $[14^37^3/3^{21}/2^{29}1^5]$
(see \cite[\S A.7]{ViDarb} for details). Comparing with the cubic covering $[2^11/3/2^11]$,
the map $\Psi_7$ branches over the three points distinguished in $[2^11/\!\star\!/\!\star1]$.
It follows that $\Psi_7$ is a Belyi map with the branching pattern  $[7^3/7^3/2^{8}1^5]$.
This can be checked by computing the divisors of $\Psi_7$ and $\Psi_7-1$ on $H_7$;
see \S \ref{sec:dc277}.

Representative $\hpgo32$-functions of types {\rm (4B)} and {\rm (7B)} are related 
by quadratic transformation (\ref{eq:t32a}). For example,
\begin{equation}
\hpg32{\!-\frac1{28},\,\frac3{14},\,\frac{13}{28}}{\frac27,\;\frac67}{-\frac{4z}{(1-z)^2}} = 
(1-z)^{-1/14}\,\hpg32{\!-\frac1{14},\,\frac1{14},\,\frac{9}{14}}{\frac27,\;\frac67}{z}\!.
\end{equation}
Consequently, the Darboux covering $\Psi_4$ 
is constructed from the fiber product
\begin{equation} \label{eq:qua47}
\frac1{\Psi_4}=-\frac{4\,\Psi_7}{(\Psi_7-1)^2}.
\end{equation}
The right-hand side is a Belyi map with the branching $[7^6/4^82^5/2^{21}]$. 
The Darboux covering is obtained by factoring out the symmetry $\Psi_7\mapsto 1/\Psi_7$.
This symmetry turns out to be the hypergeometric involution $(u,v)\mapsto (u,-v)$ on $H_7$,
as the function in (\ref{eq:qua47}) is a rational function in $u$.
The substitution $u=1/x$ gives
\begin{equation} \label{eq:darbcov4}
\Psi_4=\frac{x\,(2x^2-21+56)^2\,(x^2-28)^4\,(x^2-7x+14)^4}
{4\,(x^3-14x^2+56x-56)^7}.
\end{equation}
This is a Belyi function with the branching pattern $[7^3/4^42^21/2^{8}1^5]$.

\section{Darboux evaluations}
\label{sec:c237}

The remainder of this article gives representative Darboux evaluations 
for algebraic $\hpgo32$-functions with the projective monodromy group $\KG\cong\PSL(2,\FF_7)$
using the degree 21 Darboux coverings that reduce the monodromy to
the dihedral group $D_4$ of 8 elements.

\subsection{The case (3A)}

The Darboux covering $\Psi_3$ is given in (\ref{eq:darbcov3}). 
The point $x=0$ is (up to projective equivalence) 
a regular point after a pull-back transformation.
Accordingly, these Darboux evaluations express hypergeometric functions
as linear combinations of radical and dihedral solutions.
Note that
\begin{equation}
\Psi_3-1=\frac{49(2x-3)(4x^4-21x^2+28)\,H_1^2\,H_2^2}
{1728(x^3-7x+7)^7},
\end{equation}
where $H_1,H_2$ are degree 4 polynomials.
Let us denote
\begin{align}
P_3 = & \; 1-x+\frac{x^3}{7},\\
Q_3 = &  \left(1-\frac{3x^2}{7}\right) \!
\left(1-x+\frac{2x^2}{7}\right) \! \left(1-\frac{5x}4+\frac{11x^2}{28}\right),
\end{align}
so that 
\begin{equation}
\Psi_3=\frac{49\,x^3\,Q_3^3}{27\,P_3^7}.
\end{equation}
We use these polynomials (and normalize radical or dihedral solutions)
for local consideration at $x=0$.
\begin{theorem} \label{th:case3a}
Let us define the functions
\begin{align}
Y_0= &\, \sqrt{1-\frac{2x}3},\\
Y_1= & \left(\frac{21-8x^2+8\,\sqrt{\,7-\frac{21}4x^2+x^4}\,}{21+8\sqrt7}\right)^{\!1/4},\\
Y_2= & \left(\frac{21-8x^2-8\,\sqrt{\,7-\frac{21}4x^2+x^4}\,}{21-8\sqrt7}\right)^{\!1/4}.
\end{align}
The following identities hold in a neighborhood of $x=0$:
\begin{align} \label{eq:darbval3a}
P_3^{1/6} \; \hpg32{\!-\frac1{42},\,\frac{5}{42},\,\frac{17}{42}}{\frac13,\;\frac23}{\Psi_3} 
= &\;  \frac12\,Y_0+\frac{3+\sqrt7}{12}\,Y_1+\frac{3-\sqrt7}{12}\,Y_2,\\
 \label{eq:darbval3b}  x\,Q_3P_3^{-13/6}\,
\hpg32{\frac{13}{42},\,\frac{19}{42},\,\frac{31}{42}}{\frac23,\;\frac43}{\Psi_3}  
= & -\!3\,Y_0+\frac{3+\sqrt7}{2}\,Y_1+\frac{3-\sqrt7}{2}\,Y_2,\\
  \label{eq:darbval3c} x^2Q_3^2P_3^{-9/2} \; 
\hpg32{\frac9{14},\,\frac{11}{14},\,\frac{15}{14}}{\frac43,\;\frac53}{\Psi_3}
 = & -\!2\sqrt7\;Y_1+2\sqrt7\;Y_2.
\end{align}
\end{theorem}
\begin{proof}
A pull-back of Hurwitz equation \cite{Hurwitz86} with respect to $\Psi_3$ 
was considered in the PhD thesis \cite[\S 4.3]{ViPhD}.
The conclusion is that the functions 
\begin{equation} \label{eq:samede}
\sqrt{2x-3} \quad \mbox{and} \quad
\left(x^3-7x+7\right)^{\!1/6}\hpg32{\!-\frac1{42},\,\frac{5}{42},\,\frac{17}{42}}{\frac13,\;\frac23}{\Psi_3}
\end{equation}
satisfy the same differential equation of order 3.
The pulled-back equation \cite[(4.27)]{ViPhD} can be fully solved by 
straightforwardly applying {\sf Maple}'s command {\sf dsolve}. 
Other two solutions of the same equation are
\begin{align} \label{eq:samede2}
Y_\pm=\left(21-8x^2\pm 4\sqrt{28-21x^2+4x^4}\right)^{1/4}.
\end{align}
The functions $Y_0,Y_1,Y_2$ are the solutions $\sqrt{2x-3}$, $Y_+$, $Y_-$ 
normalized to the value 1 at $x=0$.
Formulas (\ref{eq:darbval3a})--(\ref{eq:darbval3b}) express linearly a set of
three companion $\hpgo32$-functions in the solution basis $Y_0,Y_1,Y_2$.
\end{proof}
Conversely, the solutions $Y_0,Y_1,Y_2$ can be expressed linearly in terms of hypergeometric functions.
For example,
\begin{align}
Y_0=
P_3^{1/6}\,\hpg32{\!-\frac1{42},\,\frac{5}{42},\,\frac{17}{42}}{\frac13,\;\frac23}{\Psi_3}  
 -\frac{x}6\,Q_3P_3^{-13/6}\,
\hpg32{\frac{13}{42},\,\frac{19}{42},\,\frac{31}{42}}{\frac23,\;\frac43}{\Psi_3}
\end{align}
in a neighborhood of $x=0$. Note that $Y_2=Y_1^{-1}$. 

The same hypergeometric, radical and dihedral functions can be similarly compared 
around a root of $Q_3$. For example, here are the solutions 
$Y_0,Y_1,Y_2$ normalized to have the value 1 at $x=\sqrt{7/3}$:
\begin{align}
\widetilde{Y}_0= &\, \sqrt{\big(9+2\sqrt{21}\big)(2x-3)},\\
\widetilde{Y}_1= & \left(3\cdot\frac{21-8x^2+4\,\sqrt{4x^4-21x^2+28}}
{7+4\sqrt7}\right)^{\!1/4},\\
\widetilde{Y}_2= & \left(3\cdot\frac{21-8x^2-4\,\sqrt{4x^4-21x^2+28}\,}
{7-4\sqrt7}\right)^{\!1/4}.
\end{align}
We normalize $\widetilde{P}_3 = -3 \big(9+2\sqrt{21}\big)\,P_3$ 
so that $\widetilde{P}_3\big(\sqrt{7/3}\big)=1$ as well.
Then, for instance, 
\begin{align}
3\big(\sqrt7-\sqrt3\big)\widetilde{P}_3^{\,1/6}\,
\hpg32{\!-\frac1{42},\,\frac{5}{42},\,\frac{17}{42}}{\frac13,\;\frac23}{\Psi_3}
 \hspace{-56pt} & \\ 
& = \big(2\sqrt7-3\sqrt3\big)\, \widetilde{Y}_0 
+ \frac{\sqrt7+1}{2} \, \widetilde{Y}_1 
+ \frac{\sqrt7-1}{2} \, \widetilde{Y}_2  \nonumber 
\end{align} 
in a neighborhood of $x=\sqrt{7/3}$. 

Similar expressions in terms of re-normalized $Y_0$, $Y_1$, $Y_2$ can be obtained for
the companion hypergeometric solutions
\begin{align} \!
& \widehat{Q}_3^{\,1/14}\,
\hpg32{\!-\frac1{42},\,\frac{13}{42},\,\frac{9}{14}}{\frac47,\;\frac67}{\frac1{\Psi_3}}, 
\quad  P_3\,\widehat{Q}_3^{-5/14}\,
\hpg32{\frac5{42},\,\frac{19}{42},\,\frac{11}{14}}{\frac57,\;\frac87}{\frac1{\Psi_3}}, 
\nonumber \\
& P_3^3\,\widehat{Q}_3^{-17/14}\,\hpg32{\frac{17}{42},\,\frac{31}{42},\,\frac{15}{14}}{\frac97,\;\frac{10}7}{\frac1{\Psi_3}}\;
\end{align}
around a root $x_0$ of $P_3=0$ (therefore $\Psi_3=\infty$). 
Here $\widehat{Q}_3 =Cx\,Q_3$ with a constant $C$ chosen so that $\widehat{Q}_3(x_0)=1$.
One can take $x_0=-\xi^2-2\xi+1$, where $\xi=2\cos\frac{2\pi}7$.

\subsection{The case (3B)}

By differentiating (\ref{eq:darbval3a})--(\ref{eq:darbval3c}) and using contiguous relations,
a basis of hypergeometric solutions of any differential equation (\ref{eq:hpgde}) of type (3A)
can be expressed in terms of radical and dihedral solutions.
Up to projective equivalence, the dihedral solutions will be products of $Y_1$ or $Y_2$ 
with rational functions in $\QQ\big(x,\sqrt{4x^4-21x^2+28}\big)$.
The radical solutions can be obtained by considering Riemann's $P$-symbols
of pulled-backed equations as in \cite[Proofs of Theorems 3.1, 3.3]{Vid18b}.
That is, radical solutions are constructed by picking a local exponent at each singular point
and appending a polynomial part (to match a local exponent at $x=\infty$).

As with degree 24 Darboux evaluations of type (3B) hypergeometric functions in \cite[\S 3.2]{Vid18b},
type (3B) Darboux evaluations of degree 21 appear to always require those extraneous factors to 
$Y_1,Y_2$ and $Y_0$. In particular, let
\begin{align}
W_1 & =\frac{2x^2-3x+2\,\sqrt{7-\frac{21}4x^2+x^4}}{2\sqrt7}, \\
W_2 & =\frac{3x-2x^2+2\,\sqrt{7-\frac{21}4x^2+x^4}}{2\sqrt7}.
\end{align}
Note that the conjugation of $\sqrt7$ interchanges $W_1$ and $W_2$. 
A basis of solutions of a relevant pulled-back equation is
\begin{equation}
xY_0, \qquad W_1Y_1,\qquad W_2Y_2.
\end{equation}
We obtain the following identities in a neighborhood of $x=0$: 
\begin{align}
\sqrt{P_3} \; \hpg32{\!-\frac1{14},\,\frac{3}{14},\,\frac{5}{14}}{\frac13,\;\frac23}{\Psi_3} 
= &\;  \frac{3+\sqrt7}{6}\,W_1Y_1+\frac{3-\sqrt7}{6}\,W_2Y_2,\\
x\,Q_3P_3^{-11/6}\,
\hpg32{\frac{11}{42},\,\frac{23}{42},\,\frac{29}{42}}{\frac23,\;\frac43}{\Psi_3}  
= & \,\frac12\,xY_0-\frac{\sqrt7}{6}\,W_1Y_1+\frac{\sqrt7}{6}\,W_2Y_2,\\
x^2Q_3^2P_3^{-25/6} \; 
\hpg32{\frac{25}{42},\,\frac{37}{42},\,\frac{43}{42}}{\frac43,\;\frac53}{\Psi_3} 
= & \;6\,xY_0+2\sqrt7\;W_1Y_1-2\sqrt7\;W_2Y_2.
\end{align}

\subsection{The case (4A)}
\label{sec:case4a}

The Darboux covering $\Psi_4$ is given in (\ref{eq:darbcov4}). 
The point $x=0$ is a singular point after a pull-back transformation,
with no integer differences of local exponents. 
Lemma \ref{th:locexps} implies that a basis of companion hypergeometric solutions
is matched bijectively (up to a constant factor) with the radical and dihedral solutions.
For shorthand, let us introduce the polynomials
\begin{align} \label{eq:derbp4}
P_4=& \; \textstyle 1-x+\frac14\,x^2-\frac1{56}\,x^3,\\
Q_4=& \textstyle \left(1-\frac{1}{28}\,x^2\right) \! \left(1-\frac12\,x+\frac{1}{14}\,x^2\right),\\
Q_8=& \textstyle
\left(1+3x-\frac94x^2+\frac12x^3-\frac1{28}x^4\right) \!
\left(1-\frac58x+\frac3{16}x^2-\frac1{32}x^3+\frac1{448}x^4\right),  \\
R_1=& \;\textstyle 1-\frac{3}8\,x+\frac{1}{28}\,x^2, \\
 \label{eq:derbr2} R_2=& \;\textstyle
1-\frac{33}{32}\,x+\frac{13}{22}\,x^2-\frac{9}{128}\,x^3+\frac{1}{224}\,x^4.
\end{align}
Then
\begin{equation}
\Psi_4=-\frac{343\,x\,R_1^2\,Q_4^4}{32\,P_4^7}=1-\frac{R_2\,Q_8^2}{P_4^7}.
\end{equation}
Let us also define the functions
\begin{align}
K_1 = &\, \left( { \textstyle
\frac12\sqrt{R_1R_2}+\frac12-\frac{45}{128}x+\frac5{64}x^2-\frac5{896}x^3 } \right)^{1/4}, \\[2pt]
K_2 = &\, \left( \!
\frac{\sqrt{R_1R_2}-1+\frac{45}{64}x-\frac5{32}x^2+\frac5{448}x^3}{\frac{625}{57344}\,x^2}
\right)^{\!1/4}\!.
\end{align}
They both have expansions $1+O(x)$ around $x=0$.
The conjugation $\sqrt{R_1R_2}\mapsto-\sqrt{R_1R_2}$ of $K_1$ 
equals $\frac58\,(-x^2/28)^{1/4}\,K_2$.
\begin{theorem}
The following identities hold in a neighborhood of $x=0$:
\begin{align}
\!\!\! \hpg32{\!-\frac1{28},\,\frac3{28},\,\frac{19}{28}}{\frac12,\;\frac34}{\Psi_4}
\! & = \frac{K_1}{P_4^{\,1/4}},\\  \label{eq:derb2}
\hpg32{\frac3{14},\,\frac5{14},\,\frac{13}{14}}{\frac34,\;\frac54}{\Psi_4}
\! & = \frac{P_4^{\,3/2}}{Q_4\,\sqrt{R_1}},\\
\hpg32{\frac{13}{28},\,\frac{17}{28},\,\frac{33}{28}}{\frac{5}4,\;\frac{3}2}{\Psi_4}
\! & = \frac{P_4^{\,13/4}K_2}{Q_4^2\,R_1}.
\end{align}
\end{theorem}
\begin{proof}
The third order linear differential equation for 
\begin{equation} \label{eq:hpgt4}
P_4^{1/4}\,\hpg32{\!-\frac1{28},\,\frac3{28},\,\frac{19}{28}}{\frac12,\;\frac34}{\Psi_4}.
\end{equation}
is satisfied by the following 3 functions:
\begin{align} \label{eq:hpgt4s}
x^{1/4}, \qquad 
\textstyle  \left( \pm \sqrt{R_1R_2}+1-\frac{45}{64}x+\frac5{32}x^2-\frac5{448}x^3\right)^{\!1/4}. 
\end{align}
This can be established by deriving the differential equation for the 3 functions,
and then checking it for (\ref{eq:hpgt4}). Alternatively, deriving the equation for (\ref{eq:hpgt4}) is lengthier,
but {\sf Maple}'s routine {\sf dsolve} provides the solutions in (\ref{eq:hpgt4s}).  
\end{proof}
Surely, the more attractive are purely radical Darboux evaluations as in (\ref{eq:derb2}).
Here are two more examples:
\begin{align}
\hpg32{\!-\frac1{14},\,\frac3{14},\,\frac{5}{14}}{\frac14,\;\frac34}{\Psi_4}
\! & = \frac{\sqrt{R_1}}{\sqrt{P_4}},\\
\hpg32{\frac3{14},\,\frac{5}{14},\,\frac{13}{14}}{\frac14,\;\frac34}{\Psi_4}
\! & = \frac{P_4^{3/2}\,R_1}{Q_8}\left( 1-\frac14\,x^2+\frac1{28}\,x^3 \right).
\end{align}
They can be found rather quickly by dividing out a befitting $\hpgo32$-function
by finitely many possibilities of predictable (by Riemann's $P$-symbols) powers of 
the irreducible polynomials in (\ref{eq:derbp4})--(\ref{eq:derbr2}),
and checking which combination gives the power series that appears 
to be a polynomial of predictable degree. 

\subsection{The case (4B)}

As with degree 24 Darboux evaluations of type (4B) hypergeometric functions in \cite[\S 6]{Vid18b},
the Darboux evaluations of degree 21 appear to always require extraneous factors to 
``basic" radical or dihedral expressions.
The following identities (around $x=0$) are obtained after lengthy computations and simplifications:
\begin{align}
\hpg32{\!-\frac3{28},\,\frac1{28},\,\frac{9}{28}}{\frac14,\;\frac12}{\Psi_4}
\! = & \; \frac{\frac12\sqrt{R_1R_2}+\frac12-\frac{15}{32}x+\frac{53}{448}x^2-\frac1{112}x^3}
{P_4^{\,3/4}K_1}, \\
\hpg32{\frac{11}{28},\,\frac{15}{28},\,\frac{23}{28}}{\frac{3}4,\;\frac{3}2}{\Psi_4}
\! = & \; \frac{P_4^{\,11/4}}{Q_4^2R_1}  \,
\frac{(\sqrt{R_1R_2}-1+\frac{15}{16}x-\frac{53}{224}x^2+\frac1{56}x^3)}{\frac{15}{64}\,x\,K_2}, \\
\hpg32{\frac9{14},\,\frac{11}{14},\,\frac{15}{14}}{\frac54,\;\frac74}{\Psi_4}
\! = & \; \frac{P_4^{\,9/2}}{Q_4^3R_1^{3/2}} \left( 1-\frac{3}{14}\,x \right).
\end{align}
Here are two other examples of the more attractive radical evaluations:
\begin{align}
\hpg32{\frac1{14},\,\frac{9}{14},\,\frac{11}{14}}{\frac34,\;\frac54}{\Psi_4}
\! = & \; \frac{\sqrt{P_4}}{Q_4\sqrt{R_1}} \left( 1-\frac{3}{5}x+\frac18x^2-\frac1{112}x^3 \right), \\
\hpg32{-\frac3{14},\,\frac{1}{14},\,\frac{9}{14}}{\frac14,\;\frac34}{\Psi_4}
\! = & \; \frac{\sqrt{R_1}}{P_4^{3/2}} \left( 1-\frac{3}{4}x+\frac14x^2-\frac3{112}x^3 \right).
\end{align}

\subsection{The cases (7A) and (7B)}
\label{sec:dc277}

As discussed in \S \ref{sec:ourcovs},  the Darboux curve for the types (7A), (7B)
is the genus 2 curve $H_7$ given by 
\begin{equation} \label{eq:darbc7a} \textstyle
v^2=u \, \big(1-\frac{63}4u+91u^2-231u^2+224u^4 \big).
\end{equation}
The degree 21 Darboux covering $\Psi_7$ is given by formulas (\ref{eq:darbcov7})--(\ref{eq:poly47}). 
Its principal divisor on $H_7$ is
\begin{align}
\mbox{div}(\Psi_7)=7U_1+7U_2+7U_3-7\widetilde{U}_1-7\widetilde{U}_2-7\widetilde{U}_3,
\end{align}
where the $u$-coordinates of $U_1,U_2,U_3$ satisfy $14u(2u-1)^2=1$.
The points $\widetilde{U}_1,\widetilde{U}_2,\widetilde{U}_3$ 
are obtained by the hyperelliptic involution $(u,v)\mapsto(u,-v)$.
This follows from the following divisors of the polynomial components:
\begin{align} \nonumber
\mbox{div}(2v-3u+14u^2) = & \;\textstyle (0,0)+(\frac14,-\frac1{16})+U_1+U_2+U_3-5{\cal O},\\
\mbox{div}(2v+3u-14u^2) = & \;\textstyle (0,0)+(\frac14,\frac1{16})
+\widetilde{U}_1+\widetilde{U}_2+\widetilde{U}_3-5{\cal O},\\
\mbox{div}(P_7)= & \;\textstyle (0,0)+7\,(\frac14,\frac1{16})-8{\cal O}. \nonumber
\end{align}
The Darboux or dihedral evaluations have to be evaluated at one of the points 
$U_1,U_2,U_3,\widetilde{U}_1,\widetilde{U}_2,\widetilde{U}_3$.
They are defined over $\QQ(\cos\frac{2\pi}7)$.
Therefore we settle for the following proposition. 
Handy Darboux evaluations of degree 24 for the types (7A), (7B) are given in \cite[\S 5]{Vid18b}.
\begin{propose} 
The following four functions $W_7,W_0,W_+,W_-$ on the hyperelliptic curve $H_7$ 
satisfy the same linear differential equation of order $3$:
\begin{align}
W_7 & =\sqrt{2v+3u-14u^2}\;P_7^{\,1/14}\,u^{-2/7}\,
\hpg32{\!-\frac1{14},\frac1{14},\frac5{14}}{\frac17,\,\frac57}{\Psi_7},\\
W_0 & =4u-1,\\ \label{eq:w7}
W_{\pm}  & = \left( P_1 \pm Q_1\,\sqrt{56u-21+\frac2u}\, \right)^{1/4}
\end{align}
with
\begin{align*}
P_1= &\,\textstyle 1344u^5-1540u^4+784u^3-204u^2+\frac{53}2u-\frac{11}8-3(28u^2-10u+1)v,\\
Q_1= &\,\textstyle -112u^4+84u^3-25u^2+\frac52u+(16u^2-4u+1)v.
\end{align*}
\end{propose} 
\begin{proof} 
From the proof of Theorem \ref{th:case3a} we have the four functions in 
formulas (\ref{eq:samede})--(\ref{eq:samede2}) satisfying the same linear differential equation of order 3.
By switching the $\hpgo32$-function to a companion function,
we conclude that
\begin{equation} 
\sqrt{2x-3}, \  Y_+, \  Y_-, \ \Psi_3^{1/42} \left(x^3-7x+7\right)^{\!1/6}
\hpg32{\!-\frac1{42},\,\frac{13}{42},\,\frac{9}{14}}{\frac47,\;\frac67}{\frac1{\Psi_3}}
\end{equation}
satisfy the same equation.
We use quadratic transformation (\ref{eq:hpg37c}) and the parametrization (\ref{eq:cubpara37})
to arrive at the claimed functions after a lengthy simplification.
(We can recognize $R_1,R_2$ of \S \ref{sec:case4a} in (\ref{eq:w7}), (\ref{eq:darbc7a}) 
after the substitution $u=1/x$.)
\end{proof}

\small 

\bibliographystyle{alpha}

\end{document}